\documentclass[12pt]{article}
\usepackage{amsmath,amsthm,amssymb}
\usepackage{amssymb,latexsym}
\usepackage{mathrsfs}
\newtheorem{theorem}{Theorem}

\newtheorem{corollary}{Corollary}
\newtheorem{lemma}{Lemma}

\begin{document}

\title{\bf A Diophantine inequality involving different powers of primes of the form {\boldmath$[n^c]$}}

\author{\bf S. I. Dimitrov}

\date{}

\maketitle

\begin{abstract}
Let $[\, x\,]$  denote the integer part of a real number $x$.
Assume that $\lambda_1, \lambda_2, \lambda_3$ are nonzero real numbers, not all of the same sign, that $\lambda_1/\lambda_2$ is irrational, and that $\eta$ is real. 
Let $\frac{219}{220}<\gamma<1$ and $\theta>0$.
We establish that, there exist  infinitely many triples of primes $p_1,\, p_2,\, p_3$ satisfying the inequality
\begin{equation*}
|\lambda_1p_1 + \lambda_2p_2 + \lambda_3p^4_3+\eta|<\big(\max \{p_1, p_2, p^4_3\}\big)^{\frac{219-220\gamma}{208}+\theta}
\end{equation*}
and such that $p_i=[n_i^{1/\gamma}]$, $i=1,\,2,\,3$.
\quad\\
{\bf  2020 Math.\ Subject Classification}: 11D75  $\cdot$  11P32
\end{abstract}

\section{Introduction and statement of the result}
\indent

Diophantine equations and inequalities with three prime numbers are a primary object of study in analytic number theory.
The common feature in the approach to solving them is the circle method.
Let $\lambda_1, \lambda_2, \lambda_3$ be non-zero real numbers, not all of the same sign, $\lambda_1/\lambda_2$ irrational and $\eta$ real.
We consider the Diophantine inequality
\begin{equation}\label{kth}
|\lambda_1p_1+\lambda_2p_2+\lambda_3p^k_3+\eta|<\big(\max \{p_1, p_2, p^k_3\}\big)^{-\omega+\delta}\,,
\end{equation}
where $\omega>0$, $\delta>0$ and $k\geq1$ is an integer.
The solvability of inequality \eqref{kth} for infinitely many prime triples $p_1,\,p_2,\,p_3$ has already been proved.
After extensive research, the best known results to date for $k\geq1$ are as follows:
for $k=1$, Matom\"{a}ki \cite{Mato} established $\omega=\frac{2}{9}$; for $k=2$, Gambini, Languasco and Zaccagnini \cite{Gambini} proved $\omega=\frac{1}{12}$;
for $k=3$,  Gambini, Languasco and Zaccagnini \cite{Gambini} obtained $\omega=\frac{1}{24}$; for $4\leq k\leq5$, Mu and Qu \cite{Mu} proved $\omega=\frac{5}{6k2^k}$,
and for $k\geq6$, Mu and Qu \cite{Mu} established $\omega=\frac{20}{21k2^k}$.

An additional interesting problem involves the study of Diophantine inequalities with primes of a special type.
In 1953, Piatetski-Shapiro \cite{Shapiro1953} showed that for any fixed $\frac{11}{12}<\gamma<1$, there are infinitely many primes of the form $p = [n^{1/\gamma}]$.
Such primes are referred to as Piatetski-Shapiro primes of type $\gamma$. 
Subsequently, several authors sharpened the permissible range of $\gamma$, with the best known result to date due to Rivat and Wu \cite{Rivat-Wu} with $\frac{205}{243}<\gamma<1$.
In 2022, the author \cite{Dimitrov2022} studied inequality \eqref{kth} for $k=1$ over the set of Piatetski–Shapiro primes.
More precisely, we proved that, for any fixed $\frac{37}{38}<\gamma<1$, the inequality \eqref{kth} is solvable with infinitely many primes $p_i=[n_i^{1/\gamma}]$, $i=1,\,2,\,3$. 
Very recently, the author established the solvability of inequality \eqref{kth} with Piatetski-Shapiro primes, for $k=2$ in \cite{Dimitrov2025b} and $k=3$ in \cite{Dimitrov2026}.
Motivated by these results, we consider inequality \eqref{kth} for $k=4$ with Piatetski-Shapiro primes.  
\begin{theorem}\label{Theorem}
Suppose that $\lambda_1, \lambda_2,\lambda_3$ are nonzero real numbers, not all of the same sign, that $\lambda_1/\lambda_2$ is irrational, and that $\eta$ is real. 
Let $\frac{219}{220}<\gamma<1$ and $\theta>0$.
Then there exist infinitely many ordered triples of Piatetski-Shapiro primes $p_1,\,p_2,\,p_3$ of type $\gamma$ such that
\begin{equation*}
|\lambda_1p_1+\lambda_2p_2+\lambda_3p^4_3+\eta|<\big(\max \{p_1, p_2, p^4_3\}\big)^{\frac{219-220\gamma}{208}+\theta}\,.
\end{equation*}
\end{theorem}

\section{Notations}
\indent

The letter $p$ will always denote a prime number. By $\delta$ we denote an arbitrarily small positive number, not necessarily the same in different occurrences.
Further $[t]$ and $\{t\}$ denote the integer part and the fractional part of $t$, respectively.
As usual, $\tau _k(n)$ denotes the number of solutions of the equation $m_1m_2\cdots m_k$ $=n$ in natural numbers $m_1,\ldots,m_k$.
In addition, we write $\tau (n)=\tau_2(n)$.
Moreover $\psi(t)=\{t\}-\frac{1}{2}$ and $e(t)=e^{2\pi it}$.
Let $\gamma$, $\theta$ and $\lambda_0$ be a real constants such that $\frac{219}{220}<\gamma<1$, $\theta>0$ and $0<\lambda_0<1$.
Since $\lambda_1/\lambda_2$ is irrational, there are infinitely many different convergents
$a_0/q_0$ to its continued fraction, with $\big|\frac{\lambda_1}{\lambda_2} - \frac{a_0}{q_0}\big|<\frac{1}{q_0^2}\,, (a_0, q_0) = 1\,, a_0\neq0$ and $q_0$ is arbitrary large.
Denote
\begin{align}
\label{X}
&X=q_0^\frac{13}{6}\,;\\
\label{Delta}
&\Delta=X^{-\frac{12}{13}}\log X\,;\\
\label{varepsilon}
&\varepsilon=X^{\frac{219-220\gamma}{208}+\theta}\,;\\
\label{H}
&H=\frac{\log^2X}{\varepsilon}\,;\\
\label{Sk}
&S_k(t)=\sum\limits_{\lambda_0X<p^k\leq X\atop{p=[n^{1/\gamma}]}}p^{1-\gamma}e(t p^k)\log p,;\\
\label{Sigma}
&\Sigma_k(t)=\sum\limits_{\lambda_0X<p^k\leq X}e(t p^k)\log p\,;\\
\label{U}
&U_k(t)=\sum\limits_{\lambda_0X<n^k\leq X}e(t n^k)\,;\\
\label{Omega}
&\Omega_k(t)=\sum\limits_{\lambda_0X<p^k\leq X}p^{1-\gamma}\big(\psi(-(p+1)^\gamma)-\psi(-p^\gamma)\big)e(t p^k)\log p\,;
\end{align}

\begin{align}
\label{Ik}
&I_k(t)=\int\limits_{(\lambda_0X)^\frac{1}{k}}^{X^\frac{1}{k}}e(t y^k)\,dy\,.
\end{align}

\section{Auxiliary lemmas}
\indent

\begin{lemma}\label{Fourier} Let $\varepsilon>0$ and $k\in \mathbb{N}$.
There exists a function $\theta(y)$ which is $k$ times continuously differentiable and
such that
\begin{align*}
&\theta(y)=1\hspace{12.5mm}\mbox{for }\hspace{5mm}|y|\leq 3\varepsilon/4\,;\\
&0<\theta(y)<1\hspace{5mm}\mbox{for}\hspace{7mm}3\varepsilon/4 <|y|< \varepsilon\,;\\
&\theta(y)=0\hspace{12.5mm}\mbox{for}\hspace{7mm}|y|\geq \varepsilon\,.
\end{align*}
and its Fourier transform
\begin{equation*}
\Theta(x)=\int\limits_{-\infty}^{\infty}\theta(y)e(-xy)dy
\end{equation*}
satisfies the inequality
\begin{equation*}
|\Theta(x)|\leq\min\bigg(\frac{7\varepsilon}{4},\frac{1}{\pi|x|},\frac{1}{\pi |x|}
\bigg(\frac{k}{2\pi |x|\varepsilon/8}\bigg)^k\bigg)\,.
\end{equation*}
\end{lemma}
\begin{proof}
See (\cite{Shapiro1952}).
\end{proof}

\begin{lemma}\label{Shapiroasymp} For any fixed $\frac{2426}{2817}<\gamma<1$, we have
\begin{equation*}
\sum\limits_{p\leq X\atop{p=[n^{1/\gamma}]}}1\sim \frac{X^\gamma}{\log X}\,.
\end{equation*}
\end{lemma}
\begin{proof}
See (\cite{Rivat-Sargos}, Theorem 1).
\end{proof}

\begin{lemma}\label{intSintI} We have
\begin{align*}
&\emph{(i)}\quad\quad\quad\int\limits_{-\Delta}^\Delta|S_1(t)|^2\,dt\ll X\log^3X\,,\quad\quad\int\limits_{0}^1|S_1(t)|^2\,dt\ll X^{2-\gamma}\log X\,,\\
&\emph{(ii)}\quad\quad\quad\int\limits_{-\Delta}^\Delta|I_1( t)|^2\,dt\ll X\,.
\end{align*}
\end{lemma}
\begin{proof} 
For (i) see (\cite{Dimitrov2022}, Lemma 6). For (ii) see (\cite{Vaughan1974}, Lemma 8).
\end{proof}

\begin{lemma}\label{S1I1asymptotic} Let $|t|\leq\Delta$ and $\frac{11}{12}<\gamma<1$. Then the asymptotic formula
\begin{equation*}
S_1(t)=\gamma  I_1(t)+ \mathcal{O}\left(\frac{X}{e^{(\log X)^{1/5}}}\right)
\end{equation*}
holds.
\end{lemma}
\begin{proof}
See (\cite{Dimitrov2022}, Lemma 5).
\end{proof}

\begin{lemma}\label{Omegaest} Let $\frac{99}{100}<\gamma<1$. Then
\begin{equation*}
\Omega_4(t)\ll X^{\frac{149-50\gamma}{398}+\delta}\,.
\end{equation*}
\end{lemma}
\begin{proof}
See (\cite{Dimitrov2025a}, Lemma 9).
\end{proof}

\begin{lemma}\label{Languasco} Let $k\geq1$ and $1/2X\leq Y\leq 1/2X^{1-\frac{5}{6k}+\delta}$. Then there exists a positive constant $c_1(\delta)$, which does not depend on $k$, such that
\begin{equation*}
\int\limits_{-Y}^Y\big|\Sigma_k(t)-U_k(t)\big|^2\,dt\ll\frac{X^{\frac{2}{k}-2}\log^2X}{Y}+Y^2X+X^{\frac{2}{k}-1}\mathrm{exp}\Bigg(-c_1\bigg(\frac{\log X}{\log\log X}\bigg)^{1/3}\Bigg)   \,.
\end{equation*}
\end{lemma}
\begin{proof}
See (\cite{Gambini}, Lemma 1 and Lemma 2).
\end{proof}

\begin{lemma}\label{mathfrakSest} Let $\frac{11}{12}<\gamma<1$ and $\Delta\leq|t|\leq H$. Then there exists a sequence of real numbers $X_1,\,X_2,\ldots \to \infty $ such that
\begin{equation*}
\min\Big\{\big|S_1(\lambda_{1}t)\big|,\big|S_1(\lambda_2 t)\big|\Big\}\ll X_j^{\frac{37-12\gamma}{26}}\log^5X_j\,,\quad j=1,2,\dots\,.
\end{equation*}
\end{lemma}
\begin{proof}
See (\cite{Dimitrov2022}, Lemma 7).
\end{proof}

\begin{lemma}\label{Ikest} We have
\begin{equation*}
I_k(t)\ll X^{\frac{1}{k}-1}\min\Big(X,\, |t|^{-1}\Big)\,.
\end{equation*}
\end{lemma}
\begin{proof}
See (\cite{Titchmarsh}, Lemma 4.2).
\end{proof}

\section{Initial steps}
\indent

Consider the sum
\begin{equation}\label{Gamma}
\Gamma(X)=\sum\limits_{\lambda_0X<p_1,p_2,p^4_3\leq X\atop{p_i=[n^{1/\gamma}_i],\, i=1,2,3}}\theta(\lambda_1p_1+\lambda_2p_2+\lambda_3p^4_3+\eta)p^{1-\gamma}_1p^{1-\gamma}_2p^{1-\gamma}_3\log p_1\log p_2\log p_3\,.
\end{equation}
Using the inverse Fourier transform for the function $\theta(x)$, we have
\begin{align*}
\Gamma(X)&=\sum\limits_{\lambda_0X<p_1,p_2,p^4_3\leq X\atop{p_i=[n^{1/\gamma}_i],\, i=1,2,3}}p^{1-\gamma}_1p^{1-\gamma}_2p^{1-\gamma}_3\log p_1\log p_2\log p_3\\
&\times\int\limits_{-\infty}^{\infty}\Theta(t)e\big((\lambda_1p_1+\lambda_2p_2+\lambda_3p^4_3+\eta)t\big)\,dt\\
&=\int\limits_{-\infty}^{\infty}\Theta(t)S_1(\lambda_1t)S_1(\lambda_2t)S_4(\lambda_3t)e(\eta t)\,dt\,.
\end{align*}
We write $\Gamma(X)$ as the sum of three integrals
\begin{equation}\label{Gammadecomp}
\Gamma(X)=\Gamma_1(X)+\Gamma_2(X)+\Gamma_3(X)\,,
\end{equation}
where
\begin{align}
\label{Gamma1}
&\Gamma_1(X)=\int\limits_{|t|<\Delta}\Theta(t)S_1(\lambda_1t)S_1(\lambda_2t)S_4(\lambda_3t)e(\eta t)\,dt\,,\\
\label{Gamma2}
&\Gamma_2(X)=\int\limits_{\Delta\leq|t|\leq H}\Theta(t)S_1(\lambda_1t)S_1(\lambda_2t)S_4(\lambda_3t)e(\eta t)\,dt\,,\\
\label{Gamma3}
&\Gamma_3(X)=\int\limits_{|t|>H}\Theta(t)S_1(\lambda_1t)S_1(\lambda_1t)S_4(\lambda_3t)e(\eta t)\,dt\,.
\end{align}
We shall estimate $\Gamma_1(X),\,\Gamma_2(X)$ and $\Gamma_3(X)$, respectively, in the Sections \ref{SectionGamma1}, \ref{SectionGamma2} and \ref{SectionGamma3}.
In Section \ref{Sectionfinal} we shall finalize the proof of Theorem \ref{Theorem}.

\section{Lower bound for $\mathbf{\Gamma_1(X)}$}\label{SectionGamma1}
\indent

\begin{lemma}\label{S4asymptotic} Let $\frac{99}{100}<\gamma<1$. Then
\begin{equation*}
S_4(t)=\gamma\Sigma_4(t)+\mathcal{O}\left(X^{\frac{149-50\gamma}{398}+\delta}\right)\,.
\end{equation*}
\end{lemma}
\begin{proof}
By \eqref{Sk}, \eqref{Sigma}, \eqref{Omega} and the well-known asymptotic formula
\begin{equation*}
(p+1)^\gamma-p^\gamma=\gamma p^{\gamma-1}+\mathcal{O}\left(p^{\gamma-2}\right)
\end{equation*}
we get

\begin{align}\label{S4est1}
S_4(t)&=\sum\limits_{\lambda_0X<p^4\leq X}p^{1-\gamma}\big([-p^\gamma]-[-(p+1)^\gamma]\big)e(t p^4)\log p\nonumber\\
&=\sum\limits_{\lambda_0X<p^4\leq X}p^{1-\gamma}\big((p+1)^\gamma-p^\gamma\big)e(t p^4)\log p\nonumber\\
&+\sum\limits_{\lambda_0X<p^4\leq X}p^{1-\gamma}\big(\psi(-(p+1)^\gamma)-\psi(-p^\gamma)\big)e(t p^4)\log p\nonumber\\
&=\gamma\Sigma_4(t)+\Omega_4(t)+\mathcal{O}(1)\,.
\end{align}
In view of \eqref{S4est1} and Lemma \ref{Omegaest}, we establish the statement in the lemma.
\end{proof}
Put
\begin{equation}\label{JX}
J(X)=\gamma^3\int\limits_{|t|<\Delta}\Theta(t)I_1(\lambda_1t)I_1(\lambda_2t)I_4(\lambda_3t)e(\eta t)\,dt\,.
\end{equation}
Using \eqref{Delta}, \eqref{Sk}, \eqref{Sigma}, \eqref{U}, \eqref{Ik}, \eqref{Gamma1}, \eqref{JX}, Cauchy's inequality, 
Lemma \ref{Fourier}, Lemma \ref{intSintI}, Lemma \ref{S1I1asymptotic}, Lemma \ref{Ikest} and Lemma \ref{S4asymptotic}, we deduce
\begin{align*}
\Gamma_1(X)-J(X)
&=\gamma^2\int\limits_{|t|<\Delta}\Theta(t)\Big(S_1(\lambda_1t)-\gamma I_1(\lambda_1t)\Big)I_1(\lambda_2t)I_4(\lambda_3t)e(\eta t)\,dt\nonumber\\
&+\gamma\int\limits_{|t|<\Delta}\Theta(t)S_1(\lambda_1t)\Big(S_1(\lambda_2t)-\gamma I_1(\lambda_2t)\Big)I_4(\lambda_3t)e(\eta t)\,dt\nonumber\\
&+\int\limits_{|t|<\Delta}\Theta(t)S_1(\lambda_1t)S_1(\lambda_2t)\Big(S_4(\lambda_3t)-\gamma I_4(\lambda_3t)\Big)e(\eta t)\,dt\nonumber\\
&\ll\varepsilon \frac{X}{e^{(\log X)^{1/5}}}\left[\Bigg(\int\limits_{|t|<\Delta}\big|I_1(\lambda_2t)\big|^2\,dt\Bigg)^\frac{1}{2}\Bigg(\int\limits_{|t|<\Delta}\big|I_4(\lambda_3t)\big|^2\,dt\Bigg)^\frac{1}{2}\right.\nonumber\\
&\left.+\Bigg(\int\limits_{|t|<\Delta}\big|S_1(\lambda_1t)\big|^2\,dt\Bigg)^\frac{1}{2}\Bigg(\int\limits_{|t|<\Delta}\big|I_4(\lambda_3t)\big|^2\,dt\Bigg)^\frac{1}{2}\right]\nonumber\\
&+\varepsilon \int\limits_{|t|<\Delta}\big|S_1(\lambda_1t)\big|\big|S_1(\lambda_2t)\big|\Big|\gamma\Sigma_4(\lambda_3t)-\gamma I_4(\lambda_3t)+\mathcal{O}\left(X^{\frac{149-50\gamma}{398}+\delta}\right)\Big|\,dt\nonumber\\
&\ll\varepsilon\frac{(X\log X)^\frac{3}{2}}{e^{(\log X)^{1/5}}}\Bigg(\int\limits_{|t|<\Delta}\bigg(\frac{X^{\frac{1}{4}-1}}{|t|}\bigg)^2\,dt\Bigg)^\frac{1}{2}+\varepsilon\frac{X^\frac{5}{4}}{e^{(\log X)^{1/5}}} \nonumber\\
\end{align*}

\begin{align}\label{Gamma1-JX}
&+\varepsilon \int\limits_{|t|<\Delta}\big|S_1(\lambda_1t)\big|\big|S_1(\lambda_2t)\big|\big|\Sigma_4(\lambda_3t)-U_4(\lambda_3t)\big|\,dt\nonumber\\
&+\varepsilon \int\limits_{|t|<\Delta}\big|S_1(\lambda_1t)\big|\big|S_1(\lambda_2t)\big|\big|U_4(\lambda_3t)-I_4(\lambda_3t)\big|\,dt\nonumber\\
&\ll\varepsilon\Bigg(\frac{X^\frac{5}{4}}{e^{(\log X)^{1/5}}}+J_1+J_2\Bigg)\,,
\end{align}
say. Now  Cauchy's inequality, Lemma \ref{Shapiroasymp}, Lemma \ref{intSintI} and Lemma \ref{Languasco} yield 
\begin{align}\label{J1est}
J_1&\ll X\Bigg(\int\limits_{-\Delta}^\Delta\big|S_1(\lambda_1t)\big|^2\,dt\Bigg)^\frac{1}{2}\Bigg(\int\limits_{-\Delta}^\Delta\big|\Sigma_4(\lambda_3t)-U_4(\lambda_3t)\big|^2\,dt\Bigg)^\frac{1}{2}\ll\frac{X^\frac{5}{4}}{e^{(\log X)^{1/5}}}\,.
\end{align}
By Euler's summation formula, we have
\begin{equation}\label{I-U}
I_4(t)-U_4(t)\ll1+|t|X\,.
\end{equation}
Taking into account \eqref{I-U}, Cauchy's inequality and Lemma \ref{intSintI}, we obtain 
\begin{align}\label{J2est}
J_2&\ll (1+\Delta X)\Bigg(\int\limits_{-\Delta}^\Delta\big|S_1(\lambda_1t)\big|^2\,dt\Bigg)^\frac{1}{2}\Bigg(\int\limits_{-\Delta}^\Delta\big|S_1(\lambda_2t)\big|^2\,dt\Bigg)^\frac{1}{2}\ll\Delta X^2\log^3X\,.
\end{align}
On the other hand for the integral defined by \eqref{JX}, we get
\begin{equation}\label{JXest}
J(X)=B(X)+\Phi\,,
\end{equation}
where
\begin{equation*}
B(X)=\gamma^3\int\limits_{-\infty}^{\infty}\Theta(t)I_1(\lambda_1t)I_1(\lambda_2t)I_4(\lambda_3t)e(\eta t)\,dt
\end{equation*}
and
\begin{equation}\label{Phi}
\Phi\ll\int\limits_{\Delta}^{\infty }|\Theta(t)||I_1(\lambda_1t)I_1(\lambda_2t)I_4(\lambda_3t)|\,dt\,.
\end{equation}
Working as in (\cite{Dimitrov2015}, Lemma 4), we deduce that if $\lambda_0$ is sufficiently small, then
\begin{equation}\label{BXest}
B(X)\gg\varepsilon X^\frac{5}{4}\,.
\end{equation}
Using \eqref{Phi}, Lemma \ref{Fourier} and Lemma \ref{Ikest}, we derive
\begin{equation}\label{Phiest}
\Phi\ll\frac{\varepsilon X^\frac{1}{4}}{\Delta}\,.
\end{equation}
Bearing in mind \eqref{Delta}, \eqref{Gamma1-JX}, \eqref{J1est}, \eqref{J2est}, \eqref{JXest}, \eqref{BXest} and \eqref{Phiest}, we establish 
\begin{equation}\label{Gamma1est}
\Gamma_1(X)\gg\varepsilon X^\frac{5}{4}\,.
\end{equation}

\section{Upper bound for $\mathbf{\Gamma_2(X)}$}\label{SectionGamma2}
\indent

\begin{lemma}\label{IntS^16} Let
\begin{equation*}
\mathscr{P}=\{p \;|\; p\leq X\,, \; p= [n^{1/\gamma}]\}\,,   \qquad S(t)=\sum\limits_{p\in \mathscr{P}}e(t p^4)\,.
\end{equation*}
Then
\begin{equation*}
\int\limits_{0}^1|S(t)|^{16}\,dt\ll |\mathscr{P}|^{12+\delta}\,,
\end{equation*}
where $|\mathscr{P}|$ denotes the cardinality of $\mathscr{P}$.
\end{lemma}
\begin{proof}
Our argument is a modification of the argument of Long, Li, Zhang, and Sui (\cite{Long}, Lemma 2.2). 
On the one hand
\begin{align}\label{S^2est1}
|S(t)|^2&=\sum\limits_{p_1\in \mathscr{P}}\sum\limits_{p_2\in \mathscr{P}}e\big(t(p_1^4-p_2^4)\big)\nonumber\\
&=\sum\limits_{k}\sum\limits_{p\in \mathscr{P}\atop{p+k\in \mathscr{P}}}e\big(t((p+k)^4-p^4)\big)\nonumber\\
&=\sum\limits_{j}e(tj)\sum\limits_{k}\sum\limits_{p\in \mathscr{P}\atop{p+k\in \mathscr{P}\atop{4p^3k+6p^2k^2+4pk^3+k^4=j}}}1\nonumber\\
&=\sum\limits_{j}c_je(tj)\,,
\end{align}
where $c_j$ is the number of solutions of the equation
\begin{equation*}
k(4p^3+6p^2k+4pk^2+k^3)=j
\end{equation*}
with $p, p+k\in \mathscr{P}$ and $\sum_k1\ll|\mathscr{P}|$. Taking into account that for $j=0$ one has $k=0$ or $k=-2p$, and noting $p+k\in \mathscr{P}$, it follows that $k=0$. Thus
\begin{equation}\label{c0est}
c_0\ll|\mathscr{P}|\,. 
\end{equation}
For $j\neq0$, we have
\begin{equation}\label{cjest}
c_j\ll\tau(j)\ll\tau\big(\big|k(4p^3+6p^2k+4pk^2+k^3)\big|\big)\ll|\mathscr{P}|^\varepsilon\,.
\end{equation}
On the other hand
\begin{align}\label{S^2est2}
|S(t)|^2&=\sum\limits_{p_1\in \mathscr{P}}\sum\limits_{p_2\in \mathscr{P}}e\big(t(p_1^4-p_2^4)\big)\nonumber\\
&=\sum\limits_{j}e(tj)\mathop{\sum\limits_{p_1\in \mathscr{P}}\sum\limits_{p_2\in \mathscr{P}}}_{p_1^4-p_2^4=j}1\nonumber\\
&=\sum\limits_{j}b_je(tj)\,,
\end{align}
where $b_j$ is the number of solutions of the equation
\begin{equation*}
p_1^4-p_2^4=j
\end{equation*}
with $p_1, p_2\in \mathscr{P}$. It easy to see that
\begin{equation}\label{sumbjest}
\sum\limits_{j}b_j=|S(0)|^2=|\mathscr{P}|^2 
\end{equation}
and
\begin{equation}\label{b0est}
b_0=\int\limits_{0}^1|S(t)|^2\,dt=|\mathscr{P}|\,.
\end{equation}
Now \eqref{S^2est1} -- \eqref{b0est} imply 
\begin{equation}\label{IntS^4}
\int\limits_{0}^1|S(t)|^4\,dt=\sum\limits_{j}b_jc_j=b_0c_0+\sum\limits_{j\neq0}b_jc_j\ll|\mathscr{P}|^2+|\mathscr{P}|^\varepsilon\sum\limits_{j\neq0}b_j\ll|\mathscr{P}|^{2+\varepsilon}\,.
\end{equation}
Next, we proceed with the fourth power of $|S(t)|$. Using Cauchy's inequality, we write
\begin{align}\label{S^4est1}
|S(t)|^4&=\Bigg|\sum\limits_{k_1}\sum\limits_{p\in \mathscr{P}\atop{p+k_1\in \mathscr{P}}}e\big(t(4p^3+6p^2k_1+4pk_1^2+k_1^3)k_1\big)\Bigg|^2\nonumber\\
&\ll|\mathscr{P}|\sum\limits_{k_1}\Bigg|\sum\limits_{p\in \mathscr{P}\atop{p+k_1\in \mathscr{P}}}e\big(t(4p^3+6p^2k_1+4pk_1^2+k_1^3)k_1\big)\Bigg|^2\nonumber\\
&=|\mathscr{P}|\sum\limits_{k_1}\sum\limits_{p_1\in \mathscr{P}\atop{p_1+k_1\in \mathscr{P}}}\sum\limits_{p_2\in \mathscr{P}\atop{p_2+k_1\in \mathscr{P}}}
e\big(t(4p_1^3+6p_1^2k_1+4p_1k_1^2-4p_2^3-6p_2^2k_1-4p_2k_1^2)k_1\big)\nonumber\\
&=|\mathscr{P}|\sum\limits_{k_1}\sum\limits_{k_2}\sum\limits_{p,\, p+k_1\in \mathscr{P}\atop{p+k_2,\, p+k_1+k_2\in\mathscr{P}}}e\big(2t(6p^2+6pk_1+6pk_2+2k_1^2+3k_1k_2+2k_2^2)k_1k_2\big)\nonumber\\
&=|\mathscr{P}|\sum\limits_{j}e(tj)\sum\limits_{k_1}\sum\limits_{k_2}\sum\limits_{p,\, p+k_1\in \mathscr{P}\atop{p+k_2,\, p+k_1+k_2\in \mathscr{P}\atop{2k_1k_2(6p^2+6pk_1+6pk_2+2k_1^2+3k_1k_2+2k_2^2)=j}}}1\nonumber\\
&=|\mathscr{P}|\sum\limits_{j}c^\ast_je(tj)\,,
\end{align}
where $c^\ast_j$ is the number of solutions of the equation
\begin{equation*}
2k_1k_2(6p^2+6pk_1+6pk_2+2k_1^2+3k_1k_2+2k_2^2)=j
\end{equation*}
with $p, p+k_1, p+k_2, p+k_1+k_2 \in \mathscr{P}$, $\sum_{k_1}\ll|\mathscr{P}|$ and $\sum_{k_2}\ll|\mathscr{P}|$. 
Clearly, for $j=0$, there must hold $k_1=0$ or $k_2=0$. Otherwise, if $k_1k_2\neq0$, by noting the fact that the discriminant of the quadratic polynomial in $p$ is $-12(k_1^2+k_2^2)<0$,
one has
\begin{equation*}
6p^2+6pk_1+6pk_2+2k_1^2+3k_1k_2+2k_2^2>0\,,
\end{equation*}
which contradicts to $j=0$.
Therefore
\begin{equation}\label{cast0est}
c^\ast_0\ll|\mathscr{P}|^2\,. 
\end{equation}
For $j\neq0$, we obtain
\begin{equation}\label{castjest}
c^\ast_j\ll\tau_4(j)\ll\tau_4\big(\big|2k_1k_2(6p^2+6pk_1+6pk_2+2k_1^2+3k_1k_2+2k_2^2)\big|\big)\ll|\mathscr{P}|^\varepsilon\,.
\end{equation}
On the other hand
\begin{align}\label{S^4est2}
|S(t)|^4&=\sum\limits_{p_1\in \mathscr{P}}\sum\limits_{p_2\in \mathscr{P}}\sum\limits_{p_3\in \mathscr{P}}\sum\limits_{p_4\in \mathscr{P}}e\big(t(p_1^4+p_2^4-p_3^4-p_4^4)\big)\nonumber\\
&=\sum\limits_{j}e(tj)\mathop{\sum\limits_{p_1\in \mathscr{P}}\sum\limits_{p_2\in \mathscr{P}}\sum\limits_{p_3\in \mathscr{P}}\sum\limits_{p_4\in \mathscr{P}}}_{p_1^4+p_2^4-p_3^4-p_4^4=j}1\nonumber\\
&=\sum\limits_{j}b^\ast_je(tj)\,,
\end{align}
where $b^\ast_j$ is the number of solutions of the equation
\begin{equation*}
p_1^4+p_2^4-p_3^4-p_4^4=j
\end{equation*}
with $p_1, p_2, p_3, p_4 \in \mathscr{P}$. Apparently
\begin{equation}\label{sumbastjest}
\sum\limits_{j}b^\ast_j=|S(0)|^4=|\mathscr{P}|^4\,.
\end{equation}
In view of \eqref{IntS^4}, we get
\begin{equation}\label{bast0est}
b^\ast_0=\int\limits_{0}^1|S(t)|^4\,dt\ll|\mathscr{P}|^{2+\varepsilon}\,.
\end{equation}
Now \eqref{S^4est1} -- \eqref{bast0est} give us
\begin{align}\label{IntS^8}
\int\limits_{0}^1|S(t)|^8\,dt&\ll|\mathscr{P}|\sum\limits_{j}b^\ast_jc^\ast_j=|\mathscr{P}|b^\ast_0c^\ast_0+|\mathscr{P}|\sum\limits_{j\neq0}b^\ast_jc^\ast_j\nonumber\\
&\ll|\mathscr{P}|^{5+\varepsilon}+|\mathscr{P}|^{1+\varepsilon}\sum\limits_{j\neq0}b^\ast_j\ll|\mathscr{P}|^{5+\varepsilon}\,.
\end{align}
Next, we proceed with the eighth power of $|S(t)|$. Using \eqref{S^4est1} and Cauchy's inequality, we derive
\begin{align}\label{S^8est1}
|S(t)|^8&\ll|\mathscr{P}|^2\Bigg|\sum\limits_{k_1}\sum\limits_{k_2}\sum\limits_{p,\, p+k_1\in \mathscr{P}\atop{p+k_2,\, p+k_1+k_2\in\mathscr{P}}}e\big(2t(6p^2+6pk_1+6pk_2+2k_1^2+3k_1k_2+2k_2^2)k_1k_2\big)\Bigg|^2\nonumber\\
&\ll|\mathscr{P}|^4\sum\limits_{k_1}\sum\limits_{k_2}\Bigg|\sum\limits_{p,\, p+k_1\in \mathscr{P}\atop{p+k_2,\, p+k_1+k_2\in\mathscr{P}}}e\big(2t(6p^2+6pk_1+6pk_2+2k_1^2+3k_1k_2+2k_2^2)k_1k_2\big)\Bigg|^2\nonumber\\
&=|\mathscr{P}|^4\sum\limits_{k_1}\sum\limits_{k_2}\sum\limits_{p_1,\, p_1+k_1\in \mathscr{P}\atop{p_1+k_2,\, p_1+k_1+k_2\in\mathscr{P}}}\sum\limits_{p_2,\, p_2+k_1\in \mathscr{P}\atop{p_2+k_2,\, p_2+k_1+k_2\in\mathscr{P}}}\nonumber\\
&\times e\big(2t(6p_1^2+6p_1k_1+6p_1k_2-6p_2^2-6p_2k_1-6p_2k_2)k_1k_2\big)\nonumber\\
&=|\mathscr{P}|^4\sum\limits_{k_1}\sum\limits_{k_2}\sum\limits_{k_3}\sum\limits_{p,\, p+k_1,\, p+k_2,\, p+k_3\in \mathscr{P}\atop{p+k_1+k_2,\,p+k_1+k_3,\,p+k_2+k_3\in\mathscr{P}\atop{p+k_1+k_2+k_3\in\mathscr{P}}}}
e\big(12t(2p+k_1+k_2+k_3)k_1k_2k_3\big)\nonumber\\
&=|\mathscr{P}|^4\sum\limits_{j}e(tj)\sum\limits_{k_1}\sum\limits_{k_2}\sum\limits_{k_3}
\sum\limits_{p,\, p+k_1,\, p+k_2,\, p+k_3\in \mathscr{P}\atop{p+k_1+k_2,\,p+k_1+k_3,\,p+k_2+k_3\in\mathscr{P}\atop{p+k_1+k_2+k_3\in\mathscr{P}\atop{12k_1k_2k_3(2p+k_1+k_2+k_3)=j}}}}1\nonumber\\
&=|\mathscr{P}|^4\sum\limits_{j}\overline{c}_je(tj)\,,
\end{align}
where $\overline{c}_j$ is the number of solutions of the equation
\begin{equation*}
12k_1k_2k_3(2p+k_1+k_2+k_3)=j
\end{equation*}
with $p, p+k_1, p+k_2,  p+k_3, p+k_1+k_2, p+k_1+k_3, p+k_2+k_3, p+k_1+k_2+k_3\in \mathscr{P}$, $\sum_{k_1}\ll|\mathscr{P}|$, $\sum_{k_2}\ll|\mathscr{P}|$ and $\sum_{k_3}\ll|\mathscr{P}|$\,. 
Trivially, for $j=0$, by noting that  $p+k_1+k_2+k_3\in \mathscr{P}$, we deduce $k_1=0$ or $k_2=0$ or $k_3=0$. 
Hence
\begin{equation}\label{coverline0est}
\overline{c}_0\ll|\mathscr{P}|^3\,. 
\end{equation}
For $j\neq0$, we get
\begin{equation}\label{coverlinejest}
\overline{c}_j\ll\tau_5(j)\ll\tau_5\big(\big|12k_1k_2k_3(2p+k_1+k_2+k_3)\big|\big)\ll|\mathscr{P}|^\varepsilon\,.
\end{equation}
On the other hand
\begin{align}\label{S^8est2}
|S(t)|^8&=\sum\limits_{p_i\in \mathscr{P},\, i=1,\ldots,8}e\big(t(p_1^4+p_2^4+p_3^4+p_4^4-p_5^4-p_6^4-p_7^4-p_8^4)\big)\nonumber\\
&=\sum\limits_{j}e(tj)\mathop{\sum\limits_{p_1\in \mathscr{P}}\sum\limits_{p_2\in \mathscr{P}}\sum\limits_{p_3\in \mathscr{P}}\sum\limits_{p_4\in \mathscr{P}}
\sum\limits_{p_5\in \mathscr{P}}\sum\limits_{p_6\in \mathscr{P}}\sum\limits_{p_7\in \mathscr{P}}\sum\limits_{p_8\in \mathscr{P}}}_{p_1^4+p_2^4+p_3^4+p_4^4-p_5^4-p_6^4-p_7^4-p_8^4=j}1\nonumber\\
&=\sum\limits_{j}\overline{b}_je(tj)\,,
\end{align}
where $\overline{b}_j$ is the number of solutions of the equation
\begin{equation*}
p_1^4+p_2^4+p_3^4+p_4^4-p_5^4-p_6^4-p_7^4-p_8^4=j
\end{equation*}
with $p_1, p_2, p_3, p_4, p_5, p_6, p_7, p_8 \in \mathscr{P}$. Obviously 
\begin{equation}\label{sumboverlinejest}
\sum\limits_{j}\overline{b}_j=|S(0)|^8=|\mathscr{P}|^8\,.
\end{equation}
From \eqref{IntS^8}, we have
\begin{equation}\label{boverline0est}
\overline{b}_0=\int\limits_{0}^1|S(t)|^8\,dt\ll|\mathscr{P}|^{5+\varepsilon}\,.
\end{equation}
Now \eqref{S^8est1} -- \eqref{boverline0est} lead to
\begin{align*}
\int\limits_{0}^1|S(t)|^{16}\,dt&\ll|\mathscr{P}|^4\sum\limits_{j}\overline{b}_j\overline{c}_j=|\mathscr{P}|^4\overline{b}_0\overline{c}_0+|\mathscr{P}|^4\sum\limits_{j\neq0}\overline{b}_j\overline{c}_j\nonumber\\
&\ll|\mathscr{P}|^{12+\varepsilon}+|\mathscr{P}|^{4+\varepsilon}\sum\limits_{j\neq0}\overline{b}_j\ll|\mathscr{P}|^{12+\varepsilon}\,.
\end{align*}
This completes the proof of Lemma \ref{IntS^16}.

\end{proof}

\begin{corollary}\label{IntS4^16} From Lemma \ref{IntS^16} it follows that
\begin{equation*}
\int\limits_{0}^1|S_4(t)|^{16}\,dt\ll X^{4-\gamma+\delta}\,.
\end{equation*}
\end{corollary}
Set
\begin{equation}\label{mathfrakS}
\mathfrak{S}(t,X)=\min\Big\{\big|S_1(\lambda_{1}t)\big|,\big|S_1(\lambda_2 t)\big|\Big\}\,.
\end{equation}
In view of \eqref{Gamma2}, \eqref{mathfrakS}, Lemma \ref{Fourier} and Lemma \ref{mathfrakSest}, we obtain
\begin{align}\label{Gamma2est1}
\Gamma_2(X_j)&\ll\varepsilon\int\limits_{\Delta\leq|t|\leq H}\mathfrak{S}(t, X_j)^\frac{1}{8}\big|S_1(\lambda_1 t)\big|^\frac{7}{8}\big|S_1(\lambda_2 t)\big|\big|S_4(\lambda_3 t)\big|\,dt\nonumber\\
&+\varepsilon\int\limits_{\Delta\leq|t|\leq H}\mathfrak{S}(t, X_j)^\frac{1}{8}\big|S_1(\lambda_1 t)\big|\big|S_1(\lambda_2 t)\big|^\frac{7}{8}\big|S_4(\lambda_3 t)\big|\,dt\nonumber\\
&\ll\varepsilon X_j^{\frac{37-12\gamma}{208}+\delta}\big(\Psi_1+\Psi_2\big)\,,
\end{align}
where
\begin{align}
\label{Psi1}
&\Psi_1=\int\limits_{\Delta}^H\big|S_1(\lambda_1 t)\big|^\frac{7}{8}\big|S_1(\lambda_2 t)\big|\big|S_4(\lambda_3 t)\big|\,dt\,,\\
&\Psi_2=\int\limits_{\Delta}^H\big|S_1(\lambda_1 t)\big|\big|S_1(\lambda_2 t)\big|^\frac{7}{8}\big|S_4(\lambda_3 t)\big|\,dt\,.\nonumber
\end{align}
We estimate only $\Psi_1$ and the estimation of $\Psi_2$ proceeds in the same way. Now \eqref{Psi1} and Hölder's inequality imply
\begin{equation}\label{Psi1est1}
\Psi_1\ll\Bigg(\int\limits_{\Delta}^H\big|S_1(\lambda_2t)\big|^2\,dt\Bigg)^\frac{1}{2}\Bigg(\int\limits_{\Delta}^H\big|S_1(\lambda_1t)\big|^2\,dt\Bigg)^\frac{7}{16}
\Bigg(\int\limits_{\Delta}^H\big|S_4(\lambda_3 t)\big|^{16}\,dt\Bigg)^\frac{1}{16}\,.
\end{equation}
Using Lemma \ref{intSintI} (i), we get 
\begin{equation}\label{DeltaH2}
\int\limits_{\Delta}^H\big|S_1(\lambda_kt)\big|^2\,dt\ll HX_j^{2-\gamma}\log X_j\,, \quad k=1, 2\,.
\end{equation}
From \eqref{Sk} and Corollary \ref{IntS4^16}, we wave
\begin{equation}\label{DeltaH4}
\int\limits_{\Delta}^H\big|S_4(\lambda_3t)\big|^{16}\,dt\ll HX_j^{4-\gamma+\delta}\,.
\end{equation}
Now \eqref{Psi1est1} -- \eqref{DeltaH4} lead to
\begin{equation}\label{Psi1est2}
\Psi_1\ll HX_j^{\frac{17-8\gamma}{8}+\delta}\,.
\end{equation}
Combining \eqref{varepsilon}, \eqref{H}, \eqref{Gamma2est1} and \eqref{Psi1est2}, we deduce 
\begin{equation}\label{Gamma2est2}
\Gamma_2(X_j)\ll X_j^{\frac{37-12\gamma}{208}+\delta}X_j^{\frac{17-8\gamma}{8}+\delta}=X_j^{\frac{479-220\gamma}{208}+\delta}\ll\frac{\varepsilon X_j^\frac{5}{4}}{\log X_j}\,.
\end{equation}

\section{Upper bound for $\mathbf{\Gamma_3(X)}$}\label{SectionGamma3}
\indent

From \eqref{Sk}, \eqref{Gamma3}, Lemma \ref{Fourier} and Lemma \ref{Shapiroasymp}, we derive
\begin{equation}\label{Gamma3est1}
\Gamma_3(X)\ll X^3\int\limits_{H}^{\infty}\frac{1}{t}\bigg(\frac{k}{2\pi t\varepsilon/8}\bigg)^k \,dt=\frac{X^3}{k}\bigg(\frac{4k}{\pi\varepsilon H}\bigg)^k\,.
\end{equation}
Choosing $k=[\log X]$ from \eqref{H} and \eqref{Gamma3est1}, we obtain
\begin{equation}\label{Gamma3est}
\Gamma_3(X)\ll1\,.
\end{equation}

\section{Proof of the Theorem}\label{Sectionfinal}
\indent

Summarizing  \eqref{varepsilon}, \eqref{Gammadecomp}, \eqref{Gamma1est}, \eqref{Gamma2est2} and \eqref{Gamma3est}, we get 
\begin{equation*}
\Gamma(X_j)\gg\varepsilon X_j^\frac{5}{4}=X_j^{\frac{479-220\gamma}{208}+\theta}\,.
\end{equation*}
The last estimation implies
\begin{equation}\label{Lowerbound}
\Gamma(X_j) \rightarrow\infty \quad \mbox{ as } \quad X_j\rightarrow\infty\,.
\end{equation}
Bearing in mind \eqref{Gamma} and \eqref{Lowerbound} we establish Theorem \ref{Theorem}.

\vskip30pt
\footnotesize
\begin{flushleft}
S. I. Dimitrov\\
\quad\\
Faculty of Applied Mathematics and Informatics\\
Technical University of Sofia \\
Blvd. St. Kliment Ohridski 8 \\
Sofia 1000, Bulgaria\\
e-mail: sdimitrov@tu-sofia.bg\\
\end{flushleft}

\begin{flushleft}
Department of Bioinformatics and Mathematical Modelling\\
Institute of Biophysics and Biomedical Engineering\\
Bulgarian Academy of Sciences\\
Acad. G. Bonchev Str. Bl. 105, Sofia 1113, Bulgaria \\
e-mail: xyzstoyan@gmail.com\\
\end{flushleft}


\begin{thebibliography}{}


\bibitem{Dimitrov2015} S. I. Dimitrov, T. Todorova, {\it Diophantine approximation by prime numbers of a special form},
Annuaire Univ. Sofia, Fac. Math. Inform., {\bf102}, (2015), 71 -- 90.

\bibitem{Dimitrov2022} S. I. Dimitrov, {\it Diophantine approximation by Piatetski-Shapiro primes},
Indian J. Pure Appl. Math., \textbf{53}, 4, (2022), 875 -- 883,

\bibitem{Dimitrov2025a} S. I. Dimitrov, M. Lazarova,  {\it On the distribution of $\alpha p^4$ modulo one over a thin set of primes},
Ramanujan J., \textbf{68}, 2, (2025), Art. 42.

\bibitem{Dimitrov2025b} S. I. Dimitrov, {\it Diophantine approximation with mixed powers of Piatetski-Shapiro primes},
arXiv: 2512.09771, to appear in Sib. Math. J., (2026). 

\bibitem{Dimitrov2026} S. I. Dimitrov, {\it Diophantine inequality with unlike powers of Piatetski-Shapiro primes},
hal-05437492. 

\bibitem{Gambini} A. Gambini, A. Languasco, A. Zaccagnini, {\it A Diophantine approximation problem with two primes and one $k$-th power of a prime}, 
J. Number Theory, \textbf{188}, (2018), 210 -- 228.

\bibitem{Long} L. Long, J. Li, M. Zhang, Y. Sui, {\it Cubic Waring-Goldbach problem with Piatetski-Shapiro primes}, 
arXiv: 2511.07164.

\bibitem{Mato} K. Matom\"{a}ki, {\it Diophantine approximation by primes},
Glasgow Math. J., {\bf 52}, (2010), 87 -- 106.

\bibitem{Mu} Q. Mu, Y. Qu, {\it A Diophantine inequality with prime variables and mixed power},
Acta Math. Sinica (Chin. Ser.), {\bf 58}, (2015), 491 -- 500.

\bibitem{Shapiro1952} I. I. Piatetski-Shapiro, {\it On a variant of the Waring-Goldbach problem},
Mat. Sb., {\bf30}, (1952), 105 -- 120, (in Russian).

\bibitem{Shapiro1953} I. I. Piatetski-Shapiro,
{\it On the distribution of prime numbers in sequences of the form $[f(n)]$},
Mat. Sb., {\bf 33}, (1953), 559 -- 566.

\bibitem{Rivat-Sargos} J. Rivat, P. Sargos, {\it Nombres premiers de la forme $[n^c]$},
Canad. J. Math., \textbf{53}, (2001), 414 -- 433.

\bibitem{Rivat-Wu} J. Rivat, J. Wu, {\it Prime numbers of the form $[n^c]$},
Glasg. Math. J, {\bf 43}, (2001), 237 -- 254.

\bibitem{Titchmarsh} E. Titchmarsh, {\it The Theory of the Riemann Zeta-function} 
(revised by D. R. Heath-Brown), Clarendon Press, Oxford (1986).

\bibitem{Vaughan1974} R. C. Vaughan, {\it Diophantine approximation by prime numbers I},
Proc. Lond. Math. Soc. {\bf28}, (1974), 373 -- 384.

\end{thebibliography}
\end{document}